\documentclass[letterpaper,11pt]{amsart}
\usepackage{amsmath,amssymb,amsxtra,amsthm, amstext, amscd,amsfonts,fancyhdr, hyperref,enumerate,xcolor,comment,mathtools}
\usepackage[OT2,T1]{fontenc}
\DeclareSymbolFont{cyrletters}{OT2}{wncyr}{m}{n}
\DeclareMathSymbol{\Sha}{\mathalpha}{cyrletters}{"58}
\newcommand{\bC}{{\mathbb{C}}}

\newcommand{\bQ}{{\mathbb{Q}}}
\newcommand{\bR}{{\mathbb{R}}}

\newcommand{\bZ}{{\mathbb{Z}}}

\newcommand{\Bx}{{\mathbf{x}}}

\newcommand{\G}{{\mathcal{G}}}

\renewcommand{\L}{{\mathcal{L}}}

\renewcommand{\O}{{\mathcal{O}}}

\newcommand{\Gal}{\operatorname{Gal}}

\newcommand{\GL}{\operatorname{GL}}
\newcommand{\PGL}{\operatorname{PGL}}

\newcommand{\Sym}{\operatorname{Sym}}

\newcommand{\ol}{\overline}

\newcommand{\upchi}{{\raise.35ex\hbox{$\chi$}}}

\newcommand{\SL}{\operatorname{SL}}

\makeatletter
\@namedef{subjclassname@2010}{%
	\textup{2010} Mathematics Subject Classification}
\makeatother

\newtheorem{theorem}{Theorem}[section]
\newtheorem{corollary}[theorem]{Corollary}
\newtheorem{proposition}[theorem]{Proposition}
\newtheorem{lemma}[theorem]{Lemma}

\theoremstyle{definition}

\newtheorem{question}[theorem]{Question}

\numberwithin{equation}{section}

\begin{document}
	
	\title{Covariants and simultaneous diagonalization of pairs of ternary quadratic forms, and binary quartic forms}
	\dedicatory{Dedicated to the occasion of Trevor D.~Wooley's 60th birthday.}
	\thanks{The author is supported by NSERC Discovery Grant RGPIN-2024-06810.} 
	
	\author{Stanley Yao Xiao}
	\address{Department of Mathematics and Statistics \\
		University of Northern British Columbia \\
		3333 University Way \\
		Prince George, British Columbia, Canada \\  V2N 4Z9}
	\email{StanleyYao.Xiao@unbc.ca}
	\indent
	
	%%%%%%%%%%%%%%%%%%%%%%%%%%%%%%%%%%%%%%%%%%%%%%%%%%%%%%%%%%%%%%%%%%%
	
	\begin{abstract}
	In this paper we prove a correspondence between a canonical degree six covariant of binary quartic forms $F$ and a cubic covariant of a pair of ternary quadratic forms $(f_A, f_B)$. In the process we obtain a canonical way to diagonalize a pair of $n$-ary quadratic forms over any field $K$ of characteristic zero. As a corollary, we give a precise criterion to decide whether a pair of $n$-ary quadratic forms over $\bQ$ is diagonalizable over $\bQ$. 
	\end{abstract}
	
	\maketitle

	%%%%%%%%%%%%%%%%%%%%%%%%%%%%%%%%%%%%%%%%%%
	\section{Introduction}
	\label{Intro}
	%%%%%%%%%%%%%%%%%%%%%%%%%%%%%%%%%%%%%%%%%%
	
Perhaps one of the most well-known stories in number theory in the past two decades is the one told primarily by Manjul Bhargava. Among Bhargava's achievements is the parametrization \cite{Bha2} and enumeration \cite{Bha3} of quartic rings and fields (ordered by absolute discriminant). Here Bhargava's key insight is that integral orbits of pairs of integral ternary quadratic forms correspond to quartic rings, as well as an identified cubic resolvent ring. In subsequent work with Arul Shankar \cite{BhaSha}, they used a related parametrization of $2$-Selmer elements of elliptic curves by $\GL_2(\bZ)$-equivalence classes of integral binary quartic forms to show that the average rank of elliptic curves in short Weierstrass model ordered by naive height is at most $3/2$, a stunning achievement that marked an epochal shift. \\

One of the key ingredients of Bhargava and Shankar's paper \cite{BhaSha} is the following embedding result. Let 
\begin{equation} \label{V4} V_4(\bZ) = \{a_4 x^4 + a_3 x^3 y + a_2 x^2 y^2 + a_1 xy^3 + a_0 y^4 : a_i \in \bZ, i = 0, 1, 2, 3, 4\}
\end{equation}
be the lattice of integral binary quartic forms and 
\begin{equation} W_4(\bZ) = \left\{\left(\begin{bmatrix} a_{11} & a_{12}/2 & a_{13}/2 \\ a_{12}/2 & a_{22} & a_{23}/2 \\ a_{13}/2 & a_{23}/2 & a_{33} \end{bmatrix}, \begin{bmatrix} b_{11} & b_{12}/2 & b_{13}/2 \\ b_{12}/2 & b_{22} & b_{23}/2 \\ b_{13}/2 & b_{23}/2 & b_{33} \end{bmatrix} \right): a_{ij}, b_{ij} \in \bZ,1 \leq i, j \leq 3 \right\}
\end{equation}
be the lattice corresponding to Gram matrices of pairs of ternary quadratic forms with integer coefficients. We also note the notation $W_4(\bZ) = (\bZ^2 \otimes \Sym^2 \bZ^3)^\ast$ in \cite{Bha3}. Then M.~Matchett Wood proved in \cite{Wood} that there is a canonical embedding $\phi: V_4(\bZ) \hookrightarrow W_4(\bZ)$ given by 
\begin{equation} \label{phiemb} \phi : F(x,y) = a_4 x^4 + a_3 x^3 y + a_2 x^2 y^2 + a_1 xy^3 + a_0 y^4
\end{equation}
\[\mapsto \left(\begin{bmatrix} 0 & 0 & 1/2 \\ 0 & - 1 & 0 \\ 1/2 & 0 & 0 \end{bmatrix}, \begin{bmatrix} a_4 & a_3/2 & 0 \\ a_3/2 & a_2 & a_1/2 \\ 0 & a_1/2 & a_0 \end{bmatrix} \right) = (A_0, B_F).\]
This map is canonical in the sense that the action of $\PGL_2(\bZ)$ on $V_4(\bZ)$ is realized in the group $\GL_2 (\bZ) \times \SL_3(\bZ)$ by the mapping 
\begin{equation} \rho : \PGL_2(\bZ) \rightarrow \GL_2(\bZ) \times \SL_3(\bZ) 
\end{equation}
\[\begin{bmatrix} a_{11} & a_{12} \\ a_{21} & a_{22} \end{bmatrix} \mapsto \frac{1}{a_{11} a_{22} - a_{12} a_{21}} \begin{bmatrix} a_{22}^2 & a_{21} a_{22} & a_{22}^2 \\ 2 a_{12} a_{22} & a_{11} a_{22} + a_{12} a_{21} & 2 a_{11} a_{21} \\ a_{12}^2 & a_{11} a_{12} & a_{11}^2 \end{bmatrix}.\]
(see Equation (29) in \cite{BhaSha}). \\

Importantly, the map $\phi$ in (\ref{phiemb}) is \emph{discriminant preserving}. This is crucial, because the prehomogeneous vector space $F^2 \otimes \Sym^2 F^3$ with the action of $\GL_2 \times \GL_3$ has just one polynomial discriminant, namely the discriminant. This crucial property was needed by Bhargava and Shankar to carry out their proof in \cite{BhaSha}. \\

To relate the rational space of binary quartic forms $V_4(\bQ)$ to $2$-Selmer elements of elliptic curves over $\bQ$, Bhargava and Shankar had consulted the important work of Cremona in \cite{Cre}. There Cremona studied the theory of invariants and covariants of binary cubic and quartic forms extensively, with the motivation of carrying out reduction for such forms. \\

It turns out that reduction theory of binary forms is important for solving \emph{Thue equations} (see \cite{Stew} for an authoritative modern reference). Motivated by this relation, I carried out an extensive study of the relation of invariants and covariants of binary cubic and quartic forms with their $\GL_2$-automorphism groups in \cite{X1}. This study eventually led to the joint work \cite{SX} with C.~L.~Stewart. \\

Indeed, the problem we studied in \cite{SX} was partially motivated by extending the work of C.~Skinner and T.~D.~Wooley on sums of two $k$-th powers \cite{SW} and subsequent work of Bennett, Dummigan, and Wooley \cite{BDW}. \\

In \cite{BhaSha} there does not seem to be a necessity to rely on the structure of \emph{covariants} of binary quartic forms nor pairs of ternary quadratics for their argument. Indeed, the theory of $\GL_3$-covariants of pairs of quadratic forms seemed rather mysterious. A common occurrence in the subject of arithmetic statistics is that a brilliant parametrization or identity is thought to be recently discovered, only for the community to realize that it had been known to some old masters in days of yore. This is such an occurrence, as we note that the structure of covariants of pairs of $n$-ary quadratic forms had been discussed nearly a century ago by J.~Williamson in \cite{Will}. \\

Our goal in this paper is to show that the embedding (\ref{phiemb}) not only gives a discriminant preserving map, which establishes a connection between the rings of polynomial invariants of $V_4(\bZ)$ and $W_4(\bZ)$ respectively, but also gives a connection to \emph{covariants}. The ring of covariants of $V_4$ under $\GL_2$-action is generated by $F$, the original form, $H_F$, the Hessian covariant, and 
\[F_6(x,y) = \frac{1}{36} \begin{vmatrix} \dfrac{\partial F}{\partial x} & & \dfrac{\partial F}{\partial y} \\ \\ \dfrac{\partial H_F}{\partial x} & & \dfrac{\partial H_F}{\partial y} \end{vmatrix},\]
which as far as I know is not given any particular name. We will simply refer to it as the ``$F_6$-covariant". The triple $(F, H_F, F_6)$ are not independent but satisfies a syzygy; see \cite{Cre}. \\

On the other hand, as noted in \cite{Will} the ring of polynomial covariants of $W_n$, or pairs of $n$-ary quadratic forms, under $\GL_n$-action is generated by $n$ quadratic covariants (including the two forms in the pair) as well as the Jacobian determinant of the $n$-quadratic forms. For $n = 3$, the three quadratic covariants can be taken to be 
\begin{equation} \label{maform} f_A(\Bx) = \Bx^t A \Bx, \quad f_B(\Bx) = \Bx^t B \Bx, \quad \text{and} \quad g_B(\Bx) = \Bx^t (A B^\dagger A) \Bx\end{equation} 
where $M^\dagger = \operatorname{Adj}(M)$ is the adjugate matrix of $M$. The Jacobian determinant of the three forms is denoted by $C_{3}(\Bx)$, or the \emph{cubicovariant} of $(A,B)$. \footnote{``While normally deprecating the use of portmanteau or blended words, we have the authority of Cayley, Salmon, and others for this nomenclature" - C.~Hooley, \cite{Hoo}} \\

Our main result is the following observation: 

\begin{theorem} \label{MT} Let 
\begin{equation} \label{binform} F(x,y) = a_4 x^4 + a_3 x^3 y + a_2 x^2 y^2 + a_1 xy^3 + a_0 y^4\end{equation} 
be a binary quartic form defined over a field $K$ and let $(A_0, B_F) = \phi(F)$, with $\phi$ given in (\ref{phiemb}). Let $F_6$ be the $F_6$-covariant of $F$ and let $C_3$ be the cubicovariant of the pair $(A_0, B_F) \in W_4(F)$. Then 
\begin{equation} F_6(x,y) =  C_3(x^2, xy, y^2)
\end{equation}
holds identically. 
\end{theorem}

A related result, which relates the Hessian covariant $F_4$ of $F$ with a linear combination of quadratic covariants of $(A_0, B_F)$, is the following. For $(A,B) \in W_4$ put 
\[g_A(\Bx) = \Bx^t (B A^\dagger B) \Bx.\]
We then have:
\begin{theorem} \label{MT2} Let 
\[F(x,y) = a_4 x^4 + a_3 x^3 y + a_2 x^2 y^2 + a_1 xy^3 + a_0 y^4\]
be a binary quartic form defined over a field $K$ and let $(A_0, B_F) = \phi(F)$, with $\phi$ given in (\ref{phiemb}). Let $F_4$ be the Hessian covariant of $F$. We then have 
\begin{equation} F_4(x,y) = 6 g_{B_F}(x^2, xy, y^2) + 3 g_{A_0} (x^2, xy, y^2).
\end{equation}
\end{theorem}

The cubicovariant $C_3$ of a pair $(A,B) \in W_4$ is known to be a \emph{decomposable form}. That is, $C_3$ splits into a product of linear forms in an algebraic closure of its field of definition. Despite being a ``folklore" result, it does not appear to be prominently featured in the existing feature (to the best of my knowledge). \\

A particular cubic decomposable form (actually, a \emph{norm form}) appeared in my paper \cite{X1}, namely the form 
\begin{equation} \G_{c_1, c_2}(u,s,t) = u^3 - 3c(s^2 - st + t^2)u + c((2c_1 - c_2) s^3 - 3(c_1 + c_2) s^2 t + 3 (2c_2 - c_1) st^2 + (2c_1 - c_2) t^3) 
\end{equation}
with $c  = c_1^2 - c_1 c_2 + c_2^2$ and $c_1, c_2 \in \bZ$. We gave a proof that $\G_{c_1, c_2}(u,s,t)$ is a decomposable form by showing that it is proportional to its Hessian in \cite{X1}. We give another proof, perhaps more enlightening, in this paper. 

\begin{theorem} \label{MT3} Let $a,b \in \bZ$ and let $\delta(a,b) = b^2 - 3ab + 9a^2$. Then the ternary cubic form 
\[G_{a,b}(x,y,z) = \]
\[x^3 - \frac{\delta(a,b) (y^2 + yz + z^2)}{3} x - \frac{\delta(a,b)((3a - 2b) y^3 + 3(6a - b) y^2 z + 3(3a + b) yz^2 - (3a - 2b)z^3)}{27} \]
is canonically $\SL_3\left[\frac{1}{3} \bZ \right]$-equivalent to the cubicovariant $C_3$ of the pair 
\[(A_{a,b}, B_{a,b}) = \left(\begin{bmatrix} 0 & 0 & 1 \\ 0 & -a & 0 \\ 1 & 0 & -b + 3a \end{bmatrix}, \begin{bmatrix} 0 & 1 & 0 \\ 1 & b & 0 \\ 0 & 0 & -a \end{bmatrix} \right)\]
\end{theorem}

Note that $G_{a,b}(u,s,-t)$ in Theorem \ref{MT3} is equal to $\G_{c_1, c_2}(u,s,t)$ upon putting $(a,b) = (c_1, 3c_2)$. \\

There is yet another important perspective on the cubicovariant $C_3$ of a pair $(A,B) \in W_4$, concerning the following basic question: 

\begin{question} \label{symdiag} Is there a way to canonically diagonalize a pair of ternary quadratic forms $(f,g)$, defined over a field $K$, over a possible finite field extension of $K$?
\end{question}

The question is a bit subtle, since even over $\bR$ it is not the case that an arbitrary pair of non-singular, linearly independent quadratic forms can be simultaneously diagonalized. A useful sufficient condition is that a pair $(f,g)$ of real quadratic forms can be simultaneously diagonalized if one of $f,g$ is positive definite. \\

We give an answer to Question \ref{symdiag} as follows. For a pair $(A,B)$ of $n \times n$ symmetric matrices, define 
\begin{equation} \label{FAB} F_{(A,B)}(x,y) = \det(Ax - By) = \prod_{j=1}^n (s_j x - t_j y).
\end{equation}

\begin{theorem} \label{nsymdiag} Let $n \geq 3$ be a positive integer, and let $A,B$ be symmetric $n \times n$ matrices over a field $K$ of characteristic zero. Let $(f_A, f_B)$ be the corresponding pair of $n$-ary quadratic forms. Then the pair $(A,B)$ can be simultaneously diagonalized over the splitting field $L$ of the binary form $F_{(A,B)}$ given by (\ref{FAB}). Moreover, up to permutation of columns there is a unique matrix $U$ defined over $L$ such that $(U^t A U, U^t B U)$ is a pair of diagonal matrices. The linear forms $U \Bx$, up to a scalar factor, are determined by the adjugates of the matrices $t_i A - s_i B, 1 \leq i \leq n$. 
\end{theorem}

We have the following nice corollary that characterizes simultaneous diagonalizability of pairs of quadratic forms over a field (without taking field extensions): 

\begin{corollary} \label{simcor} A pair $(f_A, f_B)$ of $n$-ary quadratic forms defined over a field $K$ of characteristic zero are simultaneously diagonalizable if and only if the binary form $F_{(A,B)}$ in (\ref{FAB}) splits completely over $K$. 
\end{corollary}

Corollary \ref{simcor} is, in principle at least, very easy to check, especially over $\bQ$. This is because asking whether a binary form (or equivalently, a univariate polynomial) has a \emph{linear factor} is much easier than asking if it is \emph{reducible} over $\bQ$.  

\subsection*{Notation}: We use the letters $A$ and $B$ to denote symmetric matrices with half-integer entries on the off-diagonal and integer entries on the diagonal. For a symmetric matrix $M$ we put $f_M(\Bx) = \Bx^t M \Bx$ to denote the associated quadratic form. For a square matrix $M$, the notation $M^\dagger$ denotes the adjugate matrix of $M$.

%%%%%%%%%%%%%%%%%%%%%%
\section{Geometry and algebra of a pair of quadratic forms} 

We start with the following theorem from \cite{Will}, which seems to have been lost in time:  

\begin{proposition} Let $W_4(K) = K^2 \otimes \Sym^2 K^3$. Then the ring of polynomial covariants of $W_4$ are generated by, in terms of a generic element $(A,B) \in W_4(K)$, the terms
\[A,\quad B, \quad A B^\dagger A,\]
and the Jacobian of $(A,B, A B^\dagger A)$. The latter is the cubicovariant $C_3$. 
\end{proposition}

Using this proposition, and noting that $C_3$ is a covariant, we prove the following well-known fact:

\begin{lemma} \label{decom} The cubicovariant $C_3$ is a decomposable form. 
\end{lemma}

\begin{proof} Since it has been established in \cite{Will} that $C_3$ is indeed a covariant, and that there is only one open orbit of $W_4(\bC)$ under the action of $G(\bC) = \GL_2(\bC) \times \GL_3(\bC)$ corresponding to the non-vanishing of the discriminant, it suffices to prove the claim when $(A,B)$ consists of a pair of diagonal forms, say 
\[(A,B ) = \left(\begin{bmatrix} s_1 & 0 & 0 \\ 0 & s_2 & 0 \\ 0 & 0 & s_3 \end{bmatrix}, \begin{bmatrix} t_1 & 0 & 0 \\ 0 & t_2 & 0 \\ 0 & 0 & t_3 \end{bmatrix} \right).\]
In this case we compute explicitly 
\begin{align*} AB^\dagger A & = \begin{bmatrix} s_1 & 0 & 0 \\ 0 & s_2 & 0 \\ 0 & 0 & s_3 \end{bmatrix} \begin{bmatrix} t_2 t_3 & 0 & 0 \\ 0 & t_1 t_3 & 0 \\ 0 & 0 & t_1 t_2 \end{bmatrix} \begin{bmatrix} s_1 & 0 & 0 \\ 0 & s_2 & 0 \\ 0 & 0 & s_3 \end{bmatrix} \\
& = \begin{bmatrix} s_1^2 t_2 t_3 & 0 & 0 \\ 0 & s_2^2 t_1 t_3 & 0 \\ 0 & 0 & s_3^2 t_1 t_2 \end{bmatrix}.
\end{align*}
Let $f_A, f_B, g_{B}$ be as in (\ref{maform}). We then find that 
\begin{align} C_3(x,y,z) & = \begin{vmatrix} \dfrac{\partial f_A}{\partial x} & & \dfrac{\partial f_A}{\partial y} & & \dfrac{\partial f_A}{\partial z} \\ \\
\dfrac{\partial f_B}{\partial x} & & \dfrac{\partial f_B}{\partial y} & & \dfrac{\partial f_B}{\partial z} \\ \\
\dfrac{\partial g_B}{\partial x} & & \dfrac{\partial g_B}{\partial y} & & \dfrac{\partial g_B}{\partial z}
 \end{vmatrix} \\ \notag \\
 & = 8(s_1 t_2 - s_2 t_1)(s_1 t_3 - s_3 t_1)(s_2 t_3 - s_3 t_2) xyz. \notag
\end{align}
This is manifestly decomposable, and we are done on noting that the property of being a decomposable form is invariant under $\GL_3$-action. 
\end{proof}

Lemma \ref{decom} reveals a curious truth: at least in the case of pairs of diagonal forms, it seems that the linear factors of $C_3$ are precisely those that appear in expressing $f_A, f_B$ as diagonal forms. This suggests that $C_3$ plays a role in simultaneously diagonalizing $(A,B)$, answering Question \ref{symdiag}.

\begin{proposition} \label{symdiagprop} Let $K$ be a field of characteristic zero and suppose $(A,B) \in W_4(K)$ be an element of non-zero discriminant. Put 
\[F_{(A,B)}(x,y) = \det(Ax - By) = \prod_{j=1}^3 (s_j x - t_j y) \in \ol{K}[x,y].\]
Then the matrices
\[\operatorname{Adj}(t_i t_j A^\dagger - s_i s_j B^\dagger), 1\leq i < j \leq 3\]
have rank one, and the quadratic forms 
\begin{equation} \label{doubadj} \Bx^t (\operatorname{Adj}( A^\dagger - s_i s_j B^\dagger)) \Bx = \pm 4(s_i t_k - s_k t_i)(s_j t_k - s_k t_j) s_i s_j t_i t_j \ell_{i,j}(\Bx)^2\end{equation}
for $\{i,j,k\} = \{1,2,3\}$. Suppose the linear transformation $(x,y,z) \mapsto (\ell_{1,2}(x,y,z), \ell_{1,3}(x,y,z), \ell_{2,3}(x,y,z)$ has determinant $\pm 1$. Then we have 
\begin{equation} \label{ABdiag} A(x,y,z) = \sum_{\{i,j,k\} = \{1,2,3\}} s_k \ell_{i,j} (x,y,z)^2 \quad \text{and} \quad B(x,y,z) = \sum_{\{i,j,k\} = \{1,2,3\}} t_k \ell_{i,j}(x,y,z)^2.
\end{equation}
\end{proposition} 

Note that Proposition \ref{symdiagprop} delivers what was promised for Question \ref{symdiag}: it provides an explicit way to diagonalize a pair $(A,B)$ of ternary quadratic forms, through its covariants and invariants. \\

The proof of Proposition \ref{symdiagprop} relies on a careful study of the geometry of a pair of ternary quadratic forms. Indeed, the $0$-dimensional scheme defined by the complete intersection of a pair $(f_A,f_B)$ of ternary quadratic forms is well studied in the literature; see for example \cite{Bha2} and \cite{Bha3}. \\

Less known is the following scheme. Consider the equation 
\begin{equation} \label{binroot} F_{(A,B)}(x,y) = \det(Ax - By) =(s_1 x - t_1y)(s_2 x - t_2 y)(s_3 x - t_3 y) = 0.\end{equation}
Then the matrices $s_i A - t_i B, i = 1,2,3$ are singular elements in the pencil generated by $(A,B)$, and in fact for $(A,B)$ in general position, all three elements will have rank $2$. Thus for each $i$ the quadratic form $f_i = \Bx^t (s_i A - t_i B) \Bx$ is the product of two linear forms over an algebraic closure of $K$, say $L_i^{(1)}, L_i^{(2)}$. That is, we have a decomposition of the form 
\[f_i(\Bx) = L_i^{(1)}(\Bx) L_i^{(2)}(\Bx) \quad \text{for} \quad i = 1,2,3.\]
The projective points $p_i$ defined by the intersection $L_i^{(1)}(\Bx) = L_i^{(2)} (\Bx) = 0$. The triangle with vertices $p_i, i = 1,2,3$ is known as the \emph{autopolar triangle} \cite{Bri}. \\

Using the forms $f_i, i = 1,2,3$ or the points $p_i, i = 1,2,3$, we may construct the linear forms $\ell_{i,j}$ in Proposition \ref{symdiagprop} as follows. First, using the $f_k$'s directly, we can take the adjugate of $M_{f_i}$, the Gram matrix of $f_i$, and note that its adjugate necessarily has rank $1$. Thus the quadratic forms
\begin{equation} G_{i,j}(\Bx) = \Bx^t \operatorname{Adj} (M_{f_k}) \Bx, \quad \{i,j,k\} = \{1,2,3\}
\end{equation}
are \emph{ramified}, meaning they are geometrically the square of a linear form. In fact, it is proportional to $\ell_{i,j}(\Bx)^2$ in Proposition \ref{symdiagprop}. \\

However, the above process for computing $\ell_{i,j}$ involves first computing the roots of the binary cubic form in (\ref{binroot}), we want to find another way of obtaining a rank-$1$ quadratic form which is proportional to the square of the linear forms $\ell_{i,j}$. This leads to Proposition \ref{symdiagprop}, which we now prove:

\begin{proof}[Proof of Proposition \ref{symdiagprop}] Note that, a prior, it is possible to diagonalize a pair of ternary quadratic forms which, by abuse of notation, we denote by $(A,B)$. Therefore, over an algebraically closed field $\ol{K}$ say, we may write 
\[A(x,y,z) = \sum_{k=1}^3 s_k \ell_{i,j}(x, y, z)^2 \quad \text{and} \quad B(x,y,z) = \sum_{k=1}^3 t_k \ell_{i,j}(x,y,z)^2\]
with $\{i,j,k\} = \{1,2,3\}$. \\

Note that the splitting field of $F_{(A,B})$, say $L$, is at most a degree $6$ extension of $K$. Note that both sides of (\ref{doubadj}) are defined over $L$. In particular, we conclude that the coefficients of $\ell_{i,j}(x_0, x_1, x_2)^2$ must be defined over $L$. This then implies that $\ell_{i,j}(x_0, x_1, x_2)$ is defined over a field extension $L^\prime$ of $L$ having degree at most two. \\

Suppose that $L^\prime$ is a degree $2$ extension of $L$. Then $\Gal(L^\prime/L)$ moves $\ell_{i,j}$ to a conjugate $\ell_{i,j}^\prime$. Note that $(\ell_{i,j}^\prime)^2$ is also defined over $L$, and thus must equal to $\ell_{i,j}^2$. Therefore, $\ell_{i,j}, \ell_{i,j}^\prime$ must be proportional over $\ol{K}$. In particular, $\ell_{i,j}$ is definable over $L$. \\

Further, by multiplying through by scalars, we may assume that the coefficient matrix $\begin{bmatrix} \ell_{1,2} & \ell_{1,3} & \ell_{2,3} \end{bmatrix}$ has determinant $1$. Then a straightforward computation gives the formula (\ref{doubadj}). Moreover, we also find that 
\[F_{(A,B)}(x,y) = (s_1 x - t_1 y)(s_2 x - t_2 y)(s_3 x - t_3 y)\]
with $A,B$ given in (\ref{ABdiag}). Since these identities are preserved under $\GL_3^{\pm 1}(L)$-substitution, Proposition \ref{symdiagprop} follows. 
\end{proof} 

We now prove Theorem \ref{nsymdiag}. 

\subsection{Proof of Theorem \ref{nsymdiag}} Given $(A,B) \in K^2 \otimes \Sym^2 K^n$, we form the binary $n$-ic form $F_{(A,B)}$ as in the ternary case
\[F_{(A,B)}(x,y) = \det(Ax - By).\]
Plainly, $F_{(A,B)}$ is defined over $K$. Let $L/K$ be the splitting field of $F_{(A,B)}$. We note that $\deg (L/K) \leq n!$. \\

From the factorization (\ref{FAB}), we see that the matrices 
\begin{equation} \label{lrm} t_i A - s_i B, \quad 1 \leq i \leq n
\end{equation}
have vanishing determinant. In fact, when $(A,B)$ has non-zero discriminant, each of these matrices have rank exactly equal to $n - 1$. But this implies that the adjugate matrices $\operatorname{Adj}(t_i A - s_i B), 1 \leq i \leq n$ have rank $1$, and hence are squares of linear forms $L_i(\Bx), 1 \leq i \leq n$. As in the proof of Proposition \ref{symdiagprop}, we see that $L_i(\Bx)$ is definable over $L$ for all $1 \leq i \leq n$. \\

However, these $\L_i$'s are not the $\ell_i$'s that we want. To obtain $\ell_i$ from the $\L_i$'s, we define the matrix 
\[\L = \begin{bmatrix} L_1 \\ \vdots \\ L_n \end{bmatrix}.\]
Our hypotheses guarantee that $\L$ is invertible. In particular, the matrix $\L_i$ obtained from $\L$ by deleting the $i$-th row has rank $n - 1$, and thus has a rank-one kernel. Let $\ell_i^\dagger$ be a basis of the kernel, and define $\ell_i(\Bx) = \langle \ell_i^\dagger, \Bx \rangle$. Then, up to scalar multiplication, these are the linear forms we are looking for. Since the $L_i$'s are defined over $L$, so must the $\ell_i$'s. \\ 

With this established, the proof follows from a routine calculation: start with linear forms $\ell_i(\Bx), 1 \leq i \leq n$ defined over $L$, and put 
\[f_A(\Bx) = \sum_{j=1}^n s_j \ell_j(\Bx)^2, \quad f_B(\Bx) = \sum_{j=1}^n t_j \ell_j(\Bx)^2\]
and follow the procedure described above, leading us to conclude that we recover the $\ell_j(\Bx)$ exactly, up to scalar factors. 

%%%%%%%%%%%%%%%%%%%%%%%%%%%%
\section{Proof of Theorems \ref{MT} and \ref{MT2}}

The proof of Theorem \ref{MT} is now straightforward, and involves an explicit computation. Let $F \in V_4$ be given as in (\ref{binform}). The Hessian covariant of $F$ is then given by 
\begin{align} \label{binhess} H_F & = \begin{vmatrix} \dfrac{\partial^2 F}{\partial x^2} & & \dfrac{\partial^2 F}{\partial x \partial y} \\ \\ \dfrac{\partial^2 F}{\partial x \partial y} & & \dfrac{\partial^2 F}{\partial y^2} \end{vmatrix} \\ 
& = (-9 a_3^2 + 24 a_2 a_4) x^4 + (-12 a_2 a_3 + 72 a_1 a_4)x^3 y + (-12 a_2^2 + 18 a_1 a_3 + 144 a_0 a_4) x^2 y^2 \notag \\
& + (-12 a_1 a_2 + 72 a_0 a_3)xy^3 + (-9 a_1^2 + 24 a_0 a_2) y^4. \notag
\end{align} 
Note that we did not normalize by $3$, as was done in \cite{Cre}. We then find that 
\begin{align} \label{F6cov} F_6(x,y) & = \frac{1}{36} \begin{vmatrix} \dfrac{\partial F}{\partial x} & & \dfrac{\partial F}{\partial y} \\ \\ \dfrac{\partial H_F}{\partial x} & & \dfrac{\partial H_F}{\partial y} \end{vmatrix} \\ 
& = (a_3^3 - 4 a_2 a_3 a_4 + 8 a_1 a_4^2) x^6 + 2(a_2 a_3^2 - 4a_2^2 a_4 + 2 a_1 a_3 a_4 + 16 a_0 a_4^2) x^5 y \notag \\
& + 5(a_1 a_3^2 - 4 a_1 a_2 a_4 + 8 a_0 a_3 a_4) x^4 y^2 + 20 (a_0 a_3^2 - a_1^2 a_4) x^3 y^3  - 5(a_1^2 a_3 - 4 a_0 a_2 a_3 + 8 a_0 a_1 a_4) x^2 y^4 \notag \\
& -2 (a_1^2 a_2 - 4 a_0 a_2^2 + 2 a_0 a_1 a_3 + 16 a_0^2 a_4) xy^5 - (a_1^3 - 4 a_0 a_1 a_2 + 8 a_0^2 a_3) y^6. \notag
\end{align}
On the other hand, for $(A_0, B_F)$ given in (\ref{phiemb}) we have
\begin{equation} A_0 B_F^\dagger A_0 = \frac{1}{16} \begin{bmatrix} - a_3^2 + 4 a_2 a_4 & 4 a_1 a_4 & a_1 a_3 \\ 4 a_1 a_4 & 16 a_0 a_4 & 4 a_0 a_3 \\ a_1 a_3 & 4 a_0 a_3 & - a_1^2 + 4 a_0 a_2 \end{bmatrix}.
\end{equation}
Then we find that 
\begin{align} C_3(u,v,w) & = \begin{vmatrix} \dfrac{\partial f_{A_0}}{\partial u} & & \dfrac{\partial f_{A_0}}{\partial v} & & \dfrac{\partial f_{A_0}}{\partial w} \\ \\ 
\dfrac{\partial f_{B_F}}{\partial u} & & \dfrac{\partial f_{B_F}}{\partial v} & & \dfrac{\partial f_{B_F}}{\partial w} \\ \\ 
\dfrac{\partial g_{B_F}}{\partial u} & & \dfrac{\partial g_{B_F}}{\partial v} & & \dfrac{\partial g_{B_F}}{\partial w} \end{vmatrix} \\ 
& = (a_3^3 - 4 a_2 a_3 a_4 + 8 a_1 a_4^2) u^3 + 2(a_2 a_3^2 - 4 a_2^2 a_4 + 2 a_1 a_3 a_4 + 16 a_0 a_4^2) u^2 v \notag \\
& + 4(a_1 a_3^2 - 4 a_1 a_2 a_4 + 8 a_0 a_3 a_4) uv^2 + 8 (a_0 a_3^2 - a_1^2 a_4) v^3 + (a_1 a_3^2 - 4 a_1 a_2 a_4 + 8 a_0 a_3) u^2 w \notag \\
& + 12 (a_0 a_3^2 - a_1^2 a_4) uvw - 4(a_1^2 a_3 - 4 a_0 a_2 a_3 + 8 a_0 a_1 a_4) v^2 w \notag \\
& - (a_1^2 a_3 - 4 a_0 a_2 a_3 + 8 a_0 a_1 a_4) uw^2 - 2(a_1^2 a_2 - 4 a_0 a_2^2 + 2 a_0 a_1 a_3 + 16 a_0^2 a_4) vw^2 \notag \\
& - (a_1^3 - 4 a_0 a_1 a_2 + 8 a_0^2 a_3) w^3. \notag
\end{align}
It is then routine to check that 
\[F_6(x,y) = C_3(x^2, xy, y^2),\]
as required. Similarly, we have 
\begin{align} g_{B_F} (x^2 ,xy, y^2) & = (-a_3^2 + 4 a_2 a_4) x^4 + (8a_1 a_4 + 2 a_1 a_3) x^3 y + (2 a_1 a_3 + 16 a_0 a_4) x^2 y^2 \\
& + 8 a_0 a_3 xy^3 + (-a_1^2 + 4 a_0 a_2) y^4, \notag \\
g_{A_0}(x^2, xy, y^2) & = - a_3^2 x^4 + (-4a_2 a_3 + 8 a_1 a_4)x^3 y + (-4a_2^2 + 2 a_1 a_3 + 16 a_0 a_4) x^2 y^2 \notag \\ 
& +(-4 a_1 a_2 + 8 a_0 a_3) xy^3 - a_1^2 y^4. \notag
\end{align} 
Comparing with (\ref{binhess}) we reach the conclusion of Theorem \ref{MT2}.

\end{document}